\documentclass[12pt]{amsart}

\usepackage{amssymb}

\usepackage{enumerate}

\usepackage{graphicx}

\makeatletter
\@namedef{subjclassname@2010}{%
  \textup{2010} Mathematics Subject Classification}
\makeatother

\newcommand{\R}{{\mathbb R}}

\newcommand{\la}{\lambda}

\newtheorem{thm}{Theorem}[section]

\newtheorem{lem}[thm]{Lemma}
\newtheorem{prop}[thm]{Proposition}

\theoremstyle{definition}
\newtheorem{defin}[thm]{Definition}

\numberwithin{equation}{section}

\begin{document}

%%%%% To ease editing, for IMPAN journals add:

\baselineskip=17pt

%%%%%%%%%%%

%% In the running head, replace first names by initials 
%% and give an abbreviation of the title.

\title[Multiplication of convex sets]{Multiplication of convex sets in $C(K)$ spaces}

\author[J. P. Moreno]{Jos\'{e} Pedro Moreno}
\address{Departamento de Matem{\'a}ticas\\ Facultad
de Ciencias\\ Universidad Aut{\'o}noma de Madrid\\ E-28049
Madrid \\ Spain} 
\email{josepedro.moreno@uam.es}

\author[R. Schneider]{Rolf Schneider}
\address{ Mathematisches Institut\\ Albert-Ludwigs-Universit{\"a}t\\ D-79104 Freiburg
i. Br.\\ Germany}
\email{rolf.schneider@math.uni-freiburg.de}

\date{}

\begin{abstract}
Let $C(K)$ denote the Banach algebra of continuous real functions, with the supremum norm, on a compact Hausdorff space $K$. For two subsets of $C(K)$, one can define their product by pointwise multiplication, just as the Minkowski sum of the sets is defined by pointwise addition. Our main interest is in correlations between properties of the product of closed order intervals in $C(K)$ and properties of the underlying space $K$. When $K$ is finite, the product of two intervals in $C(K)$ is always an interval. Surprisingly, the converse of this result is true for a wide class of compacta. We show that a first-countable space $K$ is finite whenever it has the property that the product of two nonnegative intervals is closed, or the property that the product of an interval with itself is convex. That some assumption on $K$ is needed, can be seen from the fact that, if $K$ is the Stone--\v{C}ech compactification of $\mathbb N$, then the product of two intervals in $C(K)$ with continuous boundary functions is always an interval. For any $K$, it is proved that the product of two  positive intervals is always an interval, and that the product of two nonnegative intervals is always convex. Finally, square roots of intervals are investigated, with results of similar type.
\end{abstract}



\maketitle

\section{Introduction}\label{sec1}

The relations between properties of the Banach space $C(K)$ of continuous real functions on a compact Hausdorff space and properties of the underlying space $K$ have been an object of study in analysis and topology for decades. The vector addition of closed convex sets in a real vector space, also called Minkowski addition, is a much investigated operation with many applications.  These two facts were a motivation to initiate in \cite{MS15} a study of Minkowski addition of closed convex sets in $C(K)$. We restricted ourselves there to particularly simple classes of convex sets, mainly intervals and intersections of balls. An {\em interval} $[f,g]$ in $C(K)$ is the set of all $h \in C(K)$ with $f\le h\le g$, where $f$ and $g$ are bounded real functions on $K$ (not necessarily elements of $C(K)$). Intervals in this sense are closed and convex, and the set of intervals is closed under Minkowski addition (a non-trivial fact). One of the aims of \cite{MS15} was to investigate to what extent familiar properties of the Minkowski addition in Euclidean spaces carry over to intervals and other convex sets in $C(K)$. As it turned out, some of these properties characterize the underlying space $K$ as being extremally disconnected. 

Now, $C(K)$ is not only a Banach space, but a Banach algebra, and this fact suggests to study not only the pointwise addition, but also the pointwise multiplication of convex sets. That the multiplication of sets in $C(K)$, even of open balls, bears some surprises and difficulties, was already evident from the papers \cite{BWW05}, \cite{Beh11a}, \cite{Beh11}, \cite{Beh16}, \cite{Kom06}, \cite{Kow10}. In \cite{Kom06}, for example, it is proved that multiplication in $C(K)$ is an open mapping if and only if $K$ has topological (covering) dimension zero. In the following, we study mainly the multiplication of intervals in $C(K)$. Let us first mention that in $\mathbb R$ (that is, when $K$ is a singleton) the algebraic product of two compact intervals $I,J$, defined by $I\cdot J:= \{ab: a\in I,\, b\in J\}$, is again an interval. In $\mathbb R^n$ (that is, when $K$ is finite) the same is true, replacing intervals by Cartesian products of intervals. Therefore, if the product of two subsets $\emptyset\not=A,B\subset C(K)$ is defined by $A\cdot B= \{h\eta: h\in A,\, \eta\in B\}$, it is natural to ask whether for two intervals $[f,g]$, $[\varphi,\psi]$ their product $[f,g]\cdot[\varphi,\psi]$ is an interval. This is true, for example, if $[f,g],[\varphi,\psi]$ satisfy $f>0$, $\varphi>0$ (Thm. \ref{Thm3.1}). Perhaps surprisingly, it becomes already false if only, say, $f>0$ and $\varphi\ge 0$. More to the point, if $K$ is first-countable and if the latter condition, together with continuity of $f,g,\varphi,\psi$, is sufficient in order that $[f,g]\cdot[\varphi,\psi]$ be closed, then $K$ must be finite (Thm. \ref{Thm3.2}). 

The condition that $K$ be first countable is needed to ensure, when $K$ is infinite, the existence of a nontrivial convergent sequence, a technical tool used in our constructions. That some assumption of this kind is necessary, is made evident by the fact that in $C(\beta\mathbb N)$ (where $\beta\mathbb N$ denotes the Stone--\v{C}ech compactification of the discrete space $\mathbb N$ of natural numbers) the product of two intervals $[f,g],[\varphi,\psi]$ with continuous functions $f,g,\varphi,\psi$ is always an interval (Prop. \ref{P1}). 

While the product of intervals in $C(K)$ for general $K$ remains unexplored at present, we exhibit more instances where first-countable spaces $K$ with desirable properties of the product of intervals in $C(K)$ must be finite (thus, the exceptional role of extremally disconnected spaces in \cite{MS15} is here played by finite spaces). While the product of two nonnegative intervals in $C(K)$ is always convex (Thm. \ref{Thm3.2a}), all products $[f,g]\cdot[f,g]$ with continuous $f,g$ in $C(K)$ with first-countable $K$ are  convex only if $K$ is finite (Thm. \ref{Thm3.3}). 

Finally, we study square roots of intervals. A closed convex set $C$ with $C\cdot C=[f,g]$ is called a {\em square root} of the interval $[f,g]$. In order that the interval $[f,g]$ have a square root which is an interval, it is necessary that $g\ge 0$ and $|f|\le g$. If $K$ is first-countable and if every interval $[f,g]$ with $-g\le f\le 0<g$ has a square root which is an interval, then $K$ is finite (Thm. \ref{Thm3.6}).

We may remark that our results, as far as they characterize finite spaces $K$ via properties of $C(K)$, are in line with some other results existing in the literature. Probably the best known of this type is the assertion that $K$ is finite if C(K) has finite dimension, or that $K$ is finite if $C(K)$ is reflexive. A less simple example is Prop. 4.3.12 in \cite{AK06}: Suppose that $K$ is metrizable. If $C(K)$ is order-complete, then $K$ is finite.

\section{Preliminaries}\label{sec2}

$K$ is a compact Hausdorff space, and $C(K)$ is the Banach space of continuous real functions on $K$ with the supremum norm $\|\cdot\|_\infty$. Besides the sum $f+g$ of functions $f,g\in C(K)$, there is also a (pointwise defined) product, denoted by $fg$. These operations carry over to subsets $A,B$ of $C(K)$, by means of 
$$ A+B=\{a+b:a\in A,\, b\in B\},\quad A\cdot B=\{ab: a\in A,\,b\in B\}.$$ 

Recall that by an {\em interval} in $C(K)$ we understand a set of the form
$$ [f,g]= \{h\in C(K):f\le h\le g\},$$
where $f,g:K\to\R$ are bounded functions (generally not elements of $C(K)$) and an inequality $f\le h$ is defined pointwise. Every interval is bounded and, as shown in \cite{MS15}, closed and convex (but may be empty). For the handling of intervals, the notion of semicontinuity is needed.

A bounded function $f:K\to R$ is {\em lower semicontinuous} (abbreviated by l.s.c.) if, for each $x\in K$,
$$ f(x)\le \liminf_{y\to x} f(y) := \sup_{U\in{\mathcal U}(x)} \inf_{y\in U\setminus\{x\}} f(y),$$
where ${\mathcal U}(x)$ denotes the system of neighbourhoods of $x$. The function $f$ is {\em upper semicontinuous} (abbreviated by u.s.c.)  if $-f$ is l.s.c. For a bounded function $f$,
the function 
$$ f^\wedge=\sup\{h\in C(K):h\le f\}$$ 
is l.s.c. and is called the {\em lower semicontinuous envelope} of $f$. The function 
$$ f^\vee=\inf\{h\in C(K):h\ge f\}$$ 
is u.s.c. and is called the {\em upper semicontinuous envelope} of $f$. 

As shown in \cite{MS15}, every interval $[f,g]$ satisfies $[f,g]=[f^\vee,g^\wedge]$, and $[f^\vee,g^\wedge]\not=\emptyset$ if and only if $f^\vee\le g^\wedge$. Thus, every nonempty interval has a {\em canonical representation}, that is, one of the form $[f,g]$ with $f$ u.s.c., $g$ l.s.c. and $f\le g$. For the following result, proved in \cite{MS15}, the canonical representation is crucial.

\begin{prop}
If $[f,g]$, $[\varphi,\psi]$ are intervals in $C(K)$ in canonical representation, then
\begin{equation}\label{2.1}
[f,g]+[\varphi,\psi]=[f+\varphi, g+\psi].
\end{equation}
\end{prop}

For intervals in $C(K)$, the properties collected in the following definition will play an important role.

\begin{defin}
An interval is called
\begin{enumerate}[\upshape (i)]
\item {\em signed} if it has a representation $[f,g]$ where either $f>0$ or $g<0$,
\item {\em positive} {\em  (nonnegative)} if it has a representation $[f,g]$ with $f>0$ (with $f\ge 0$),
\item {\em  continuous} if it has a representation $[f,g]$ with continuous functions $f,g$.
\end{enumerate}
\end{defin}

We note that for a signed interval, also its canonical representation $[f,g]$ has the property that either  $f>0$ or $g<0$. Further, a representation $[f,g]$ of an interval with continuous $f,g$ is its canonical representation. We also remark that in \cite{MS15}, Thm. 6.2, the continuous intervals have been characterized as being precisely the summands of balls.

Turning to products, we note an obvious fact on products of real intervals, which will be used repeatedly. Let $I,J\subset\R$ be nonempty, compact intervals. Their (algebraic) product is defined by
$$ I\cdot J:= \{ab: a\in I,\, b\in J\}.$$
(We call this the `algebraic' product, to distinguish it from the Cartesian product.) Let
$$ u =  \min\{ab:a\in I,\,b\in J\},\quad v =  \max\{ab:a\in I,\,b\in J\}.$$
Then $I\cdot J\subset [u,v]$, trivially. If $u=a_0b_0$, $v= a_1b_1$ with $a_0,a_1\in I$ and $ b_0,b_1\in J$ define, for $t\in[0,1]$,
$$ f(t):= \left((1-t)a_0+ta_1\right)\cdot\left((1-t)b_0+tb_1\right).$$
Then $f(0)=u$ and $f(1)=v$. Let $z\in [u,v]$. By continuity, there exists $\tau\in[0,1]$ with $f(\tau)=z$, hence
$$ z=\left((1-\tau)a_0+\tau a_1\right)\cdot\left((1-\tau)b_0+\tau b_1\right),$$
and here $(1-\tau)a_0+\tau a_1\in I$ and $(1-\tau)b_0+\tau b_1\in J$. It follows that $I\cdot J=[u,v]$.

\section{Products of intervals in $C(K)$}\label{sec3}

Recall that $K$ is a compact Hausdorff space and $C(K)$ is the Banach algebra of continuous real functions on $K$ with the supremum norm. In view of the result (\ref{2.1}) on the sum of two intervals, it is natural to ask what is true for their product,
$$
[f,g]\cdot[\varphi,\psi]=\{h\eta: h\in [f,g],\, \eta\in[\varphi,\psi]\}\,.
$$

A particular case is the product of a scalar with an interval (the scalar can be considered as a constant function). Clearly, for an arbitrary interval $[f,g]\in C(K)$, we have $\la [f,g]=[\la f, \la g]$ if $\la>0$ and $\la [f,g]=[\la g, \la f]$ if $\la<0$. The distributive law of multiplication of a scalar with respect to the Minkowski addition of two intervals works fine. By contrast, the multiplication of an interval with respect to the addition of scalars satisfies the distributive law only when the scalars have the same sign. Consider, for instance, the unit ball $[-\bar 1,\bar 1]$ of the space $C(K)$ (we denote by $\bar c$ the constant function with value $c$). Then 
$$ [-\bar 1,\bar 1]=(-1+2)[-\bar 1,\bar 1] \not= (-1)[-\bar 1,\bar 1]+2[-\bar 1,\bar 1]=[-\bar 1,\bar 1]+[-\bar 2,\bar 2]=[-\bar 3,\bar 3].$$

Suppose now that $[f,g],[\varphi,\psi]$ are intervals in $C(K)$. For each $x\in K$ we have
$$ [f(x),g(x)]\cdot[\varphi(x),\psi(x)] = [u(x),v(x)] $$
with
\begin{eqnarray*}
u(x) = \min\{ab: a\in[f(x),g(x)],\,b\in[\varphi(x),\psi(x)]\},\\
v(x) = \max\{ab: a\in[f(x),g(x)],\,b\in[\varphi(x),\psi(x)]\}.
\end{eqnarray*}
This defines bounded functions $u,v$ on $K$. Trivially, we have
$$ [f,g]\cdot [\varphi,\psi]\subset [u,v].$$
We call $(u,v)$ the {\em pair of bounding functions} corresponding to the quadruple $(f,g,\varphi,\psi)$ (although this correspondence is not shown in the notation).

\begin{lem}\label{Lem3.1}
Let $[f,g],[\varphi,\psi]$ be intervals in canonical representation in $C(K)$, and let $(u,v)$ be the pair of bounding functions corresponding to $(f,g,\varphi,\psi)$.
\begin{enumerate}[\upshape (i)]
\item If $[F,G]= [f,g]\cdot[\varphi, \psi]$ is an interval in canonical representation and $x\in K$, then
\begin{equation*}
 [F(x),G(x)]= [f(x),g(x)]\cdot[\varphi(x),\psi(x)].
\end{equation*}
\item If the product $[f,g]\cdot[\varphi, \psi]$ is an interval, then it is the interval $[u^\vee,v^\wedge]$.
\item If $f,g,\varphi,\psi$ are continuous, then $u,v$ are continuous.
\end{enumerate}
\end{lem}

\begin{proof} First we remark that for an interval $[f,g]$ in canonical representation and for $x\in K$ and $\lambda\in\R$ with $f(x)\le\lambda\le g(x)$, there exists a function $h\in[f,g]$ with $h(x)=\lambda$; see  \cite{MS15}, Lemma 4.2. Assertion (i) follows immediately from this remark.

To prove (ii), let $f,g,\varphi,\psi,u,v$ be as in the lemma. Assume that the product $P:=[f,g]\cdot[\varphi, \psi]$ is an interval. It is nonempty and hence has a unique canonical representation $P= [F,G]$. For each $H\in P$ we have $u\le H\le v$ and hence also
$$ u^\vee \le H\le v^\wedge.$$
Since $[F,G]$ is in canonical representation, we have $F=\inf\, [F,G]$ and $G=\sup\,[F,G]$ (see \cite{MS15}, p. 357) and hence $u^\vee\le F$ and $G\le v^\wedge$. Since $u(x)$ is an attained minimum, there are numbers $a,b$ with $f(x)\le a\le g(x)$, $\varphi(x)\le b\le \psi(x)$ and $ab=u(x)$. By  the initial remark, there are functions $h\in[f,g]$ and $\eta\in[\varphi,\psi]$ such that $h(x)=a$ and $\eta(x)=b$. Then $h\eta\in P=[F,G]$ and $(h\eta)(x)=ab=u(x)$. It follows that $u(x)\ge F(x)$. We conclude that $u\ge F$, hence $u^\vee\ge F^\vee = F$ and thus $u^\vee=F$. Simililary we obtain that $v^\wedge =G$. This proves (ii).

To prove (iii), let $f,g,\varphi,\psi$ be continuous. The mapping 
$$ x\mapsto [f(x),g(x)]\times[\varphi(x),\psi(x)]$$
from $K$ into the set of rectangles in $\R^2$ is continuous with respect to the Hausdorff metric on the space of nonempty compact subsets of $\R^2$. The product mapping $(a,b)\mapsto ab$ from $\R^2$ to $\R$ is continuous. Therefore, also the mapping $x\mapsto \min [f(x),g(x)]\cdot[\varphi(x),\psi(x)]$ is continuous. Similar for $\max$.
\end{proof}

We return to the question whether to the representation (\ref{2.1}) of the sum of two intervals in canonical representation there is a counterpart for the algebraic product. It turns out that such a counterpart exists only under very strong additional assumptions. First we prove the following.

\begin{thm}\label{Thm3.1}
The product of two signed intervals in $C(K)$ is always a signed interval.
\end{thm}

\begin{proof}
Let $[f,g],[\varphi,\psi]$ be intervals in $C(K)$, without loss of generality in canonical representation. We shall prove the following:
\begin{eqnarray}
&& \mbox{if}  \quad f> 0,  \, \varphi> 0,\quad \mbox{then}\quad [f,g]\cdot[\varphi,\psi] = [f\varphi, g\psi], \label{3.1}\\
&& \mbox{if}  \quad f>0,\, \psi< 0,\quad \mbox{then}  \quad [f,g]\cdot[\varphi,\psi] = [g\varphi, f\psi],\label{3.2}\\
&& \mbox{if}  \quad g< 0,\, \psi< 0,\quad \mbox{then}  \quad [f,g]\cdot[\varphi,\psi] = [g\psi, f\varphi],\label{3.3}\\
&& \mbox{if}  \quad g< 0, \,\varphi > 0,\quad \mbox{then}  \quad [f,g]\cdot[\varphi,\psi] = [f\psi, g\varphi].\label{3.4}
\end{eqnarray}

Suppose, first, that $f>0$, $\varphi> 0$. The relation $[f,g]\cdot[\varphi,\psi]\subset [f \varphi, g \psi]$ is trivial. To prove the reverse inclusion, assume that $H\in [f \varphi, g \psi]$. Let
$$ \lambda =\max\{f,H/\psi\},\quad \mu=\min\{g,H/\varphi\}.$$
Since $\psi$ is l.s.c. and $H$ is continuous, the function $H/\psi$ is u.s.c. Since also $f$ is u.s.c., the function $\lambda$ is u.s.c. Similarly, $\mu$ is l.s.c. Moreover, $f\le g$, $f\le H/\varphi$, $H/\psi\le g$, $H/\psi\le H/\varphi$, hence $\lambda\le \mu$. By the Kat\v{e}tov--Tong insertion theorem (\cite{Kat51}, \cite{Ton52}, or \cite{Eng89}, p. 61), there exists a function $h\in[\lambda,\mu]$. Then $f\le h\le g$, and the function $\eta:= H/h$ (which can be defined since $f>0$ implies $h>0$) satisfies $\varphi\le \eta\le\psi$, thus
$$ H= h\eta\in[f,g]\cdot[\varphi,\psi].$$
This completes the proof of (\ref{3.1}).

Relations (\ref{3.2})--(\ref{3.4}) follow from (\ref{3.1}). For instance, in the case of (\ref{3.2}) we have
$$ [f,g]\cdot[\varphi,\psi]= -[f,g]\cdot[-\psi,-\varphi]= -[-f\psi,-g\varphi]= [g\varphi,f\psi].$$
Similarly, the remaining cases (\ref{3.3}), (\ref{3.4}) are settled.
\end{proof}

\section{Properties of products of intervals}\label{sec4}

One might be inclined to think that, for example, in (\ref{3.1}) the positive intervals could be replaced by nonnegative intervals. This, however, is generally not possible. Of course, if the space $K$ is discrete (and thus finite, since $K$ is compact), then for any two intervals $[f,g]$, $[\varphi,\psi]$ we have $[f,g]\cdot [\varphi,\psi]=[u,v]$ (with $u,v$ as in Lemma \ref{Lem3.1}), since any $H\in [u,v]$ can be written, for each $x\in K$, in the form $H(x)=h(x)\eta(x)$ with $h(x)\in [f(x),g(x)]$ and $\eta(x)\in[\varphi(x),\psi(x)]$, and every real function on $K$ is continuous. Under a restriction on the space $K$, we can reverse this observation. As it turns out, none of the inequalities $f> 0$, $\varphi> 0$ can be relaxed in general, even if the functions $f,g,\varphi,\psi$ are continuous, and even if we do not require the full interval property of the product, but only closedness.

\begin{thm}\label{Thm3.2}
Assume that $K$ is first-countable. If in $C(K)$ the product of a positive and a nonnegative interval, both continuous, 
is always closed, then $K$ is finite.
\end{thm}

\begin{proof}
Assume that the assumptions are satisfied. Suppose that $K$ is not finite. Then (since $K$ is compact) there is a point $x_0\in K$ that is not an isolated point, that is, every neighbourhood of $x_0$ contains a point different from $x_0$. Since $K$ is a first-countable Hausdorff space, there is in $K$ a decreasing sequence $(U_n)_{n\in{\mathbb N}}$ of open sets forming a neighbourhood basis of $x_0$, and also a sequence $(y_n)_{n\in{\mathbb N}}$ of points satisfying $y_n \in U_n\setminus U_{n+1}$ for $n\in{\mathbb N}$. Then the sequence $(y_n)_{n\in{\mathbb N}}$ is injective and converges to $x_0$.

By induction, we show the existence of a sequence $(p_n)_{n\in{\mathbb N}}$ of continuous functions $p_n:K\to [1,2]$ with special properties. First, we set $p_1\equiv 1$. If for the number $n\in{\mathbb N}$ the continuous function $p_n:K\to[1,2]$ has already been defined, we use the Tietze extension theorem to find a continuous function $p_{n+1}:K\to [1,2]$  with
$$ p_{n+1}(y_{n+1})= \left\{ \begin{array}{ll} 1, & \mbox{if $n+1$ is odd},\\[2mm]  2, & \mbox{if $n+1$ is even,}\end{array}\right.\quad p_{n+1}(x)=p_n(x)\quad \mbox{for }x\in K\setminus U_{n+1}.$$
After this inductive process, to each $x\in K\setminus\{x_0\}$ there is some $m\in{\mathbb N}$ such that $x\notin U_m$ and therefore $p_k(x)=p_m(x)$ for $k\ge m$. We define $p(x)= p_m(x)$ and put $p(x_0)=0$. Summarizing, we have:

The function $p:K\to[1,2]$ is continuous on $K\setminus \{x_0\}$, and it satisfies 
$$ p(y_n)= \left\{ \begin{array}{ll} 1, & \mbox{if $n$ is odd},\\[2mm]  2, & \mbox{if $n$ is even.}\end{array}\right.$$
Next, we define
$$ \varphi(x_0)=0,\quad \varphi(y_n)= 1/n\quad\mbox{for }n\in{\mathbb N},$$
and extend this to a continuous function $\varphi:K\to[0,1]$. Further, we define $f,g:K\to\R$ by $f=\bar 1$ and $g=\bar 2$.

The function
$$ H=p\varphi$$
is continuous on $K$, since $p$ is continuous on $K\setminus\{x_0\}$, $\varphi$ is continuous on $K$, the function $p$ is bounded, and $\varphi(x_0)=0$. We claim that $H$ is in the closure of the product $[f,g]\cdot[\varphi,\varphi]$. To prove this, let $\varepsilon>0$ be given. We can choose a number $m\in{\mathbb N}$ with $\varphi(x)\le\varepsilon$ for $x\in U_m$. Since $p_m\in [f,g]$, the function $H_m = p_m\varphi$ satisfies $H_m\in [f,g]\cdot[\varphi,\varphi]$. Further, we have $H_m(x)=H(x)$ for $x\in K\setminus U_m$, while for $x\in U_m$ we get 
$$|H_m(x)-H(x)|= |p_m(x)-p(x)|\varphi(x)\le \varepsilon.$$ 
Thus, $\|H_m-H\|_\infty \le\varepsilon$. Since $\varepsilon>0$ was arbitrary, this shows that $H$ is in the closure of $[f,g] \cdot[\varphi, \varphi]$. Since by assumption the latter is closed, we have $H\in [f,g]\cdot[\varphi,\varphi]$, hence there is a function $h\in [f,g]$ with $H=h\varphi$. This yields $h(x)=p(x)$ whenever $\varphi(x)\not=0$, in particular $h(y_n)=1$ for odd $n$ and $h(y_n)=2$ for even $n$, which contradicts the continuity of $h$ at $x_0$. 
\end{proof}

In Theorem \ref{Thm3.2} (and in later theorems), first countability of $K$ is assumed in order to ensure that
for every $A\subset K$ and every $x\in\overline A$, there exists a sequence in $A$ that converges to $x$. That some assumption of this kind cannot be avoided, can be seen from the following. 

Let $\mathbb N$ denote the discrete space of natural numbers (it could be replaced by any infinite discrete space), and let $\beta\mathbb N$ be the Stone--\v{C}ech compactification of $\mathbb N$ (see, e.g., \cite{Car05}). For $K=\beta\mathbb N$, the proof of Theorem \ref{Thm3.2} breaks down, since $\beta\mathbb N$ does not contain a nontrivial convergent sequence (see, e.g., Corollary 3.6.15 in \cite{Eng89}). And indeed, in $C(\beta\mathbb N)$ the product of two intervals defined by continuous functions is always an interval. To show this, we use the fact that there is an order-preserving linear isometry $\Phi:\ell^{\infty}(\mathbb N)\to C(\beta\mathbb N)$, where $\ell^{\infty}(\mathbb N)$ is the Banach space of all bounded real functions on $\mathbb N$ with the supremum  norm. 

\begin{prop}\label{P1}
In $C(\beta\mathbb N)$, the product of any two continuous intervals is an interval.
\end{prop} 

\begin{proof}
The Stone--\v{C}ech compactification $\beta\mathbb N$ contains a dense subspace homeomorphic to $\mathbb N$, and we can assume that $\mathbb N$ is already a dense subspace of $\beta\mathbb N$. The linear isometry $\Phi$ is then the extension of a function in $\ell^\infty(\mathbb N)$ to a function in $C(\beta\mathbb N)$.

Let $[f,g]$, $[\varphi,\psi]$ be intervals in $C(\beta\mathbb N)$, where $f,g,\varphi,\psi$ are continuous, and let $(u,v)$ be the pair of bounding functions corresponding to the quadruple $(f,g,\varphi,\psi)$. According to Lemma \ref{Lem3.1}, the functions $u,v$ are also continuous. We state that
\begin{equation}\label{SC}
[f,g]\cdot[\varphi,\psi]=[u,v].
\end{equation}
The inclusion $[f,g]\cdot[\varphi,\psi]\subset [u,v]$ is clear. 

Let $H\in [u,v]$. For each $n\in\mathbb N$ we have
$$ [f(n),g(n)]\cdot[\varphi(n),\psi(n)] =[u(n),v(n)],$$
by the definition of $u,v$. Since $H(n)\in[u(n),v(n)]$, there are numbers $\tilde h(n)\in[f(n),g(n)]$ and $\tilde\eta(n)\in[\varphi(n),\psi(n)]$ such that
$$ \tilde h(n)\tilde\eta(n)=H(n).$$
The functions $\tilde h,\tilde\eta$ thus defined on $\mathbb N$ are bounded and hence elements of $\ell^\infty(\mathbb N)$. Let $h=\Phi(\tilde h)$, $\eta=\Phi(\tilde \eta)$ be their extensions to $C(\beta\mathbb N)$. Since on the dense subspace $\mathbb N$ the continuous functions $f,g,\varphi,\psi,h,\eta$ satisfy $f\le h\le g$, $\varphi\le\eta\le\psi$ and $H=h\eta$, they satisfy the same relations on $\beta\mathbb N$. This proves that $h\in [f,g]$, $\eta\in[\varphi,\psi]$, $H=h\eta$ on $\beta\mathbb N$ and thus $H\in[f,g]\cdot[\varphi,\psi]$, which completes the proof of (\ref{SC}).
\end{proof}

In contrast, in $C(\alpha{\mathbb N})$, where $\alpha{\mathbb N}$ denotes the one point compactification of $\mathbb N$, the product of two continuous intervals need not be an interval. This follows from Theorem \ref{Thm3.2}, since $\alpha\mathbb N$ is a first-countable compact Hausdorff space.

Intervals are closed and convex. As the product of two intervals $[f,g]$, $[\varphi,\psi]$ need not be closed, even if they are nonnegative, it is a natural question whether it must at least be convex. That this is not a trivial question, can be seen from Theorem \ref{Thm3.3}, where the nonnegativity assumption is deleted and convexity does not hold generally. With the nonnegativity assumption, however, convexity can be proved.

\begin{thm}\label{Thm3.2a}
The product of two nonnegative intervals in $C(K)$ is always convex.
\end{thm}

\begin{proof}
Let bounded functions $0\le f\le g$, $0\le\varphi\le\psi$ be given, and let $H_1,H_2\in[f,g] \cdot [\varphi,\psi]$. Thus, there are functions $h_1,h_2\in[f,g]$, $\eta_1,\eta_2\in[\varphi,\psi]$ such that $H_i=h_i\eta_i$ for $i=1,2$. Let $0<t<1$. We have to show that 
$$ H:=(1-t)H_1+t H_2=(1-t)h_1\eta_1+th_2\eta_2$$
satisfies
\begin{equation}\label{3.5n} 
H\in[f,g] \cdot [\varphi,\psi].
\end{equation}

To prepare this, let $M=\{(a,\alpha,b,\beta)\in \R^4: a,\alpha,b,\beta\ge 0\}$ and $z=(a,\alpha,b,\beta)\in M$. We consider $(a,\alpha)$ and $(b,\beta)$ as points in the Euclidean plane $\R^2$ with the standard basis. They are opposite vertices of a unique rectangle (possibly degenerate) with edges parallel to the coordinate axes. We denote this rectangle by ${\rm Rec}(z)$. Further, we define the hyperbola
$$ {\rm Hyp}(z) = \{(c,\gamma)\in\R^2:c\gamma=(1-t)a\alpha+tb\beta\}$$
(which may be degenerate). The function $p$ defined by
$$ p(\tau)=\big((1-\tau)a+\tau b\big)\cdot\big((1-\tau)\alpha+\tau \beta\big),\quad 0\le\tau\le 1,$$
satisfies $p(0)=a\alpha$, $p(1)=b\beta$, hence there is some $\tau_0\in [0,1]$ with $p(\tau_0)=(1-t)a\alpha+tb\beta$. Thus, the rectangle ${\rm Rec}(z)$ has a nonempty intersection with the hyperbola ${\rm Hyp}(z)$. We need to find a continuous map $P:M\to \R^2$ satisfying, for all $z\in M$,
\begin{equation*}
%\label{3.6n}  
P(z)\in {\rm Rec}(z)\cap {\rm Hyp}(z)\,.
\end{equation*}
Given such a map, we can proceed as follows. For $x\in K$, we define 
$$z(x)=(h_1(x),\eta_1(x),h_2(x),\eta_2(x))$$ 
%and
%$$ {\mathcal R}(x)= {\rm Rec}(z(x)),\qquad {\mathcal H}(x)= {\rm Hyp}(z(x)).$$
Then we define ${\mathcal P}(x)= P(z(x))$ for $x\in K$ and thus a mapping $\mathcal P: K\to \R^2$. It is continuous  and satisfies ${\mathcal P}(x)\in {\rm Rec}(z(x))\cap{\rm Hyp}(z(x))$. Hence, if we define functions $h,\eta$ on $K$ by ${\mathcal P}(x)=(h(x),\eta(x))$, then $h$ and $\eta$ are continuous, they satisfy $h(x)\eta(x)=H(x)$ since ${\mathcal P}(x)\in {\rm Hyp}(z(x))$, and $f\le h\le g$, $\varphi\le \eta\le \psi$ since ${\mathcal P}(x)\in{\rm Rec}(z(x))$ and the latter rectangle is contained in the rectangle ${\rm Rec}(f(x),g(x),\varphi(x),\psi(x))$. Thus, $h\in[f,g]$, $\eta\in [\varphi,\psi]$ and $h\eta=H$, which shows that $H\in [f,g]\cdot[\varphi,\psi]$ and thus proves (\ref{3.5n}).

We proceed to construct the mapping $P$. First, if $a=b=0$, we define
$$ P(z)= (0,(1-t)\alpha+t\beta).$$
If $\alpha=\beta=0$, we define
$$ P(z)= ((1-t)a+tb,0).$$

Now we assume that $\max\{a,b\}>0$ and $\max\{\alpha,\beta\}>0$. Then $(1-t)a+tb\not=0$ and $(1-t)\alpha+t\beta\not=0$, hence we can define
\begin{eqnarray}
A &=&\left(\frac{(1-t)a\alpha+tb\beta}{(1-t)\alpha+t\beta},\,(1-t)\alpha+t\beta\right), \label{3.7n} \\ 
B &=&\left((1-t)a+tb,\, \frac{(1-t)a\alpha+tb\beta}{(1-t)a+tb}\right).\label{3.8n}
\end{eqnarray}
The point $A$ belongs to ${\rm Hyp}(z)$ (trivially) and to ${\rm Rec}(z)$ (observe that its first coordinate is a convex combination of $a$ and $b$, and its second coordinate is a convex combination of $\alpha$ and $\beta$). Similarly, $B \in {\rm Hyp}(z)\cap{\rm Rec}(z)$. The point 
\begin{equation}\label{3.9n} 
C = \frac{((1-t)\alpha+t\beta)A+((1-t)a+tb)B}{(1-t)\alpha+t\beta+(1-t)a+tb},
\end{equation}
is a convex combination of $A$ and $B$ and hence belongs to ${\rm Rec}(z)$. Let $L$ be the line in $\R^2$ that is parallel to the vector $(1,1)$ and passes through $C$. Either the points $A$ and $B$ coincide, then $C$ is this point and hence lies on ${\rm Hyp}(z)$, or the points $A$ and $B$ lie on different sides of $L$. Since the part of ${\rm Hyp}(z)$ between $A$ and $B$ lies in ${\rm Rec}(z)$, the line $L$ intersects ${\rm Hyp}(z)$ in a unique point belonging to ${\rm Rec}(z)$. This point we define as $P(z)$.

We have to show that the mapping $P$ thus defined is continuous. Let $z=(a,\alpha,b,\beta)\in M$. Suppose, first, that $\max\{a,b\}>0$ and $\max\{\alpha,\beta\}>0$. Then there is a neighbourhood $V$ of $z$ in $M$ such that $\max\{a',b'\}>0$ and $\max\{\alpha',\beta'\}>0$ for all $z'=(a',\alpha',b',\beta')\in V$. Since $A,B,C$ defined above are continuous functions of $a,\alpha,b,\beta$ and also the intersection point of $L$ and ${\rm Hyp}$ defines a continuous function, the mapping $P$ is continuous at $z$. 

Suppose, second, that $z=(0,0,0,0)$. Then ${\rm Rec}(z)=\{(0,0)\}$ and $P(z) = (0,0)$. Let $U$ be a neighbourhood of $(0,0)$ in $\R^2$. There is a neighbourhood $V$ of $z$ in $M$ such that for all $z'\in M$ we have ${\rm Rec}(z')\subset U$ and hence $P(z')\in U$. Thus, $P$ is continuous at $z$.

Third, let $z=(a,\alpha,b,\beta)\in M$ belong to one of the remaining cases, that is, either $\alpha=\beta=0$, $\max\{a,b\}>0$, or $a=b=0$, $\max\{\alpha,\beta\}>0$. Suppose, first, that $\alpha=\beta=0$, $\max\{a,b\}>0$. By definition, $P(z)= ((1-t)a+tb,0)$. Let $U$ be a neighbourhood of $P(z)$ in $\R^2$. Let $z_i=(a_i,\alpha_i,b_i, \beta_i) \in M$, $i\in {\mathbb N}$, be such that $z_i\to z$ for $i\to\infty$. We may assume that $\max\{a_i,b_i\}>0$, since this holds for all sufficiently large $i$. If $\alpha_i=\beta_i=0$, then $P(z_i)=((1-t)a_i+tb_i,0)$, which is in $U$ for sufficiently large $i$. In the following, we consider the indices $i$ for which $(\alpha_i,\beta_i)\not=(0,0)$. First we notice that the parameter $(1-t)a_i \alpha_i+tb_i\beta_i$ of the hyperbola ${\rm Hyp}(z_i)$ tends to zero for $i\to\infty$. Therefore, there is a neighbourhood $U'\subset U$ with the following property. If $C\in U'$, then the line in $\R^2$ that is parallel to the vector $(1,1)$ and passes through $C$, intersects the hyperbola ${\rm Hyp}(z_i)$ in a point inside $U$. Define $A_i,B_i,C_i$ by (\ref{3.7n}), (\ref{3.8n}), (\ref{3.9n}) with $z$ replaced by $z_i$. Then $B_i\to ((1-t)a+tb,0)$, and $A_i$ remains bounded, since the functions $f,g,\varphi,\psi$ are bounded.  The weight of $A_i$ in
$$ C_i = \frac{((1-t)\alpha_i+t\beta_i)A_i+((1-t)a_i+tb_i)B_i}{(1-t)\alpha_i+t\beta_i+(1-t)a_i+tb_i}$$
tends to zero for $i\to\infty$, and it follows that $C_i\to ((1-t)a+tb,0)$. Therefore, $C_i\in U'$ and hence $P(z_i)\in U$ for sufficiently large $i$. This shows that $P$ is continuous at $z$. The case $a=b=0$, $\max\{\alpha,\beta\}>0$, is treated similarly.
\end{proof}

In the proof of Theorem \ref{Thm3.2}, it was important that we can take products of different intervals. In fact, the product $[f,g]\cdot[f,g]$ with $f\ge 0$ is always an interval, since for $H\in [f^2,g^2]$, the function $h=H^{1/2}$ satisfies $h\in[f,g]$ and $h^2=H$. That the assumption $f\ge 0$ cannot be deleted, is shown by the following theorem. It reveals that in general the product $[f,g]\cdot[f,g]$, even for continuous $f,g$, need not be an interval, because it need even not be convex.

\begin{thm}\label{Thm3.3} 
Assume that $K$ is first-countable. If for any continuous interval $[f,g]$ in $C(K)$ the product $[f,g]\cdot [f,g]$ is convex, then $K$ is finite.
\end{thm}

\begin{proof} Assume that the assumptions are satisfied, but $K$ is not finite. Then there is a point $x_0\in K$ which is not an isolated point. Since $K$ is a first-countable Hausdorff space, there is an injective sequence $(y_n)_{n\in{\mathbb N}}$ in $K\setminus\{x_0\}$ that converges to $x_0$. The set $X =\{y_n:n\in{\mathbb N}\}\cup\{x_0\}$ is closed. Define
$f:X\to [-2,0]$ and $g:X\to [0,2]$ by 
$$ f(y_n)=-1+(-1)^n/n,\qquad g(y_n)=1+(-1)^n/n  \quad\mbox{for }n\in{\mathbb N}$$
and $f(x_0)=-1$, $g(x_0)=1$.
Since $K$ is a normal space, there exist continuous extensions $f:K\to [-2,0]$ and $ g:K\to [0,2]$. Then $f\le g$ on $K$.

If we define
$$
H=\frac{1}{2}f^2+\frac{1}{2}g^2\,,
$$
then
$$
H(y_n)=1+\frac{1}{n^2}
$$
for every $n\in{\mathbb N}$.
Since $f^2,g^2\in[f,g]\cdot[f,g]$ and the latter is convex by assumption, we have 
$$H\in [f,g]\cdot[f,g].$$ 
Therefore, there are functions $h,\eta\in [f,g]$ such that $h\eta=H$. For each $n\in{\mathbb N}$ we have
$$ h(y_n)\eta(y_n)= 1+\frac{1}{n^2}.$$
In particular, $h(y_n)$ and $\eta(y_n)$ have the same sign. Let $n$ be even. If $h(y_n)$ and $\eta(y_n)$ are both negative, then
$$ h(y_n)\eta(y_n) \le \left(-1+\frac{1}{n}\right)^2 <1+ \frac{1}{n^2},$$
a contradiction. Therefore, $h(y_n)$ and $\eta(y_n)$ are both positive, which implies that
$$ 1+\frac{1}{n^2} = h(y_n)\eta(y_n)\le h(y_n)g(y_n) =h(y_n)\left(1+\frac{1}{n}\right)$$
and thus $h(y_n)\ge 1-1/n$. Let $n$ be odd. If $h(y_n)$ and $\eta(y_n)$ are both positive, then
$$ h(y_n)\eta(y_n) \le \left(1-\frac{1}{n}\right)^2 <1+ \frac{1}{n^2},$$
a contradiction. Therefore, $h(y_n)$ and $\eta(y_n)$ are both negative, and
$$ 1+\frac{1}{n^2} = |h(y_n)\eta(y_n)|\le |h(y_n)||f(y_n)| =|h(y_n)|\left(1+\frac{1}{n}\right),$$
hence $h(y_n)\le -1+1/n$. Thus, every neighbourhood of $x_0$ contains points where $h$ is arbitrarily close to $1$ and points where $h$ is arbitrarily close to $-1$. Since $h$ is continuous, this is a contradiction. It shows that the space $K$ is finite.
\end{proof}

\section{Square roots of intervals}\label{sec5}

A {\em convex body} in $C(K)$ is a nonempty bounded closed convex subset of $C(K)$. In \cite{MS15}, Theorem 4.3, the following was shown. If $C,D\subset C(K)$ are convex bodies such that the closure of the sum C+D is an interval, then each of $C$ and $D$ is an interval. One may ask whether an analogous result is true for products instead of sums. But it is easily seen that this is not the case (see the example below).

Now we restrict ourselves to products of a convex body with itself, and we say that the convex body $C$ is a {\em square root} of the interval $[f,g]$ if $C \cdot C=[f,g]$. It is easy to see that an interval can have a square root that is not an interval, provided that $K$ is not a singleton. For example, let $x_0,x_1$ be distinct points of $K$, let $$ A=\{f\in[-\bar 1,\bar 1]: f(x_0)\ge 0\},\qquad B=[-\bar 1,-\bar 1].$$ 
Then $A$ and $B$ are convex bodies in $C(K)$. Let $C$ be the convex hull of $A\cup B$; it is closed. Then $C\cdot C\subset [-\bar 1,\bar 1]$ trivially. Conversely, $f\in [-\bar 1,\bar 1]$ satisfies either $f(x_0)\ge 0$ and  hence $f= f\cdot \bar 1\in C\cdot C$, or $f(x_0)<0$ and hence $f=(-f)\cdot (-\bar 1)\in C\cdot C$. Thus, $C$ is a square root of $[-\bar 1,\bar 1]$. On the other hand, $C$ is not an interval. If it were an interval, it could only be the interval $[-\bar 1,\bar 1]$, but a function $h\in [-\bar 1,\bar 1]$ with $h(x_0)=-1$ and $h(x_1)=1$ (which exists by Urysohn's lemma) does not belong to $C$. 

This raises the question: which intervals have a square root that is an interval. Trivially, this holds for intervals $[f,g]$ with $f\ge 0$. In fact, if $f\ge 0$ and if $H\in[f,g]$, then $\sqrt{H}\in [\sqrt{f},\sqrt{g}]$, which shows that $[\sqrt{f},\sqrt{g}]\cdot[\sqrt{f},\sqrt{g}]=[f,g]$.

The following result provides additional information, which will be used below to show that also an assumption on the  existence of square roots leads to finite spaces. 

\begin{prop}\label{P2}
Let $[f,g]$ be an interval in canonical representation in $C(K)$. In order that $[f,g]$ have a square root which is an interval,
it is necessary that
\begin{equation}\label{3.33} 
g\ge 0\quad\mbox{ and }\quad |f|\le g\, ,
\end{equation}
and it is sufficient that $g$ is continuous and
\begin{equation}\label{3.34} 
-g\le f\le 0<g\,.
\end{equation}
\end{prop}

\begin{proof} To prove the necessity consider, first, an interval $[a,b]$ with $a \le b$ in $\R$, and suppose that it has a square root (that is, an interval $[c,d]$ with $[c,d]\cdot[c,d]=[a,b]$). Then $b\ge 0$ trivially. If $a\ge 0$, then $[\sqrt{a},\sqrt{b}]$ and $[-\sqrt{b}, -\sqrt{a}]$ are the square roots of $[a,b]$. Let $a< 0\le b$. A square root of $[a,b]$ must be of the form $[c,d]$ with $c\le 0\le d$. If $d\ge |c|$, then $[c,d]\cdot [c,d] = [cd,d^2]$, hence $[a/\sqrt{b},\sqrt{b}]$ is a square root of $[a,b]$. If $d<|c|$, then $[c,d]\cdot [c,d] = [cd,c^2]$, hence $[-\sqrt{b},-a/\sqrt{b}]$ is a square root of $[a,b]$. In either case, $[a/\sqrt{b},\sqrt{b}]$ and $[-\sqrt{b}, -a/\sqrt{b}]$ are the square roots of $[a,b]$, and we have $|a|\le b$.

Now let $[f,g]$ be an interval in canonical representation in C(K). Suppose it has a square root which is an interval. Then we can assume that the latter is also in canonical representation. Now it follows from Lemma \ref{Lem3.1}(1), by
applying the preceding observation to $[a,b]=[f(x),g(x)]$ for $x\in K$, that condition (\ref{3.33}) is necessary, as stated. 

To prove the sufficiency, suppose that (\ref{3.34}) is satisfied and that $g$ is continuous. Define, for $x\in K$, 
$$ \varphi(x)= f(x)/\sqrt{g(x)}, \quad \psi(x)=\sqrt{g(x)}\,.$$
For each $x\in K$ we have $[\varphi(x),\psi(x)] \cdot [\varphi(x),\psi(x)] = [f(x),g(x)]$. This shows that $[\varphi,\psi] \cdot[\varphi,\psi]\subset [f,g]$. If $H\in [f,g]$, the functions $h= \sqrt{g}$ and $\eta = H/\sqrt{g}$
satisfy $h,\eta\in[\varphi,\psi]$ and $h\eta=H$. It follows that $[\varphi,\psi] \cdot[\varphi,\psi]=[f,g]$.
\end{proof}

The continuity in the second part of Proposition \ref{P2} is essential, as the following theorem shows.

\begin{thm}\label{Thm3.6}
Let $K$ be first-countable. If every interval $[f,g]$ in canonical representation in $C(K)$ satisfying $(\ref{3.34})$ has a square root which is an interval, then $K$ is finite.
\end{thm}

\begin{proof}
Let $K$ be first-countable, and suppose every interval in canonical representation satisfying (\ref{3.34}) has a square root which is an interval. Assume that $K$ is not finite. Then $K$ has a point $x_0$ that is not isolated. Let $(y_n)_{n\in{\mathbb N}}$ be an injective sequence in $K\setminus\{x_0\}$ converging to $x_0$. We define $f=-1/2$ on $K$ and
$$ g(x)=\left\{\begin{array}{ll} 1/2 & \mbox{for $x=y_n$ with odd $n$ and for $x=x_0$,} \\[2mm] 1 & \mbox{for the remaining $x\in K$}. \end{array} \right.$$
Then $g$ is l.s.c., and $[f,g]$ is an interval in canonical representation satisfying (\ref{3.34}). By assumption, it has a square root which is an interval, say $[\varphi,\psi]$, without loss of generality also in canonical representation. Therefore, in particular, the function $f$ has a representation $f=h\eta$ with $h,\eta\in[\varphi,\psi]$.

Let $x=y_n$ with even $n$. Then $[f(x),g(x)]=[-1/2,1]$. As shown in the proof of Proposition \ref{P2}, we must have either $[\varphi(x),\psi(x)]= [-1/2,1]$ or $[\varphi(x),\psi(x)]= [-1,1/2]$. Since $-1/2 =f(x)=h(x)\eta(x)$, this yields
$$ h(x)\in\{\pm 1/2,\pm 1\}.$$

Let $x=y_n$ with odd $n$. Then $[f(x),g(x)]=[-1/2,1/2]$ and $[\varphi(x),\psi(x)]=[-1/\sqrt{2},1/\sqrt{2}]$. Since $-1/2 =f(x)=h(x)\eta(x)$, this yields 
$$ h(x)\in\{\pm 1/\sqrt{2}\}.$$
Thus, in every neighbourhood of $x_0$, there are points (namely $y_n$ with sufficiently large even $n$) were $h$ attains one of the values $\pm 1/2,\pm 1$, and points (namely $y_n$ with sufficiently large odd $n$) were $h$ attains one of the values $\pm 1/\sqrt{2}$. This contradicts the continuity of $h$.
\end{proof}

We remark that, in analogy to Proposition \ref{P1}, in $C(\beta\mathbb N)$ any interval $[f,g]$ with continuous functions $f,g$ satisfying (\ref{3.34}), has a square root which is an interval.

\subsection*{Acknowledgements}
This research was partly supported by Ministerio de Ciencia e Innovaci\'{o}n, grant MTM2012-34341. 

We thank the anonymous referee for a thorough report, and we highly appreciate his/her valuable suggestions for improvements.


\begin{thebibliography}{99}


\normalsize
\baselineskip=17pt

\bibitem{AK06} F. Albiac, N.J. Kalton, \emph{Topics in Banach Space Theory}. Graduate Texts in Mathematics, 233. Springer, New York, 2006.
\bibitem{BWW05} M. Balcerzak, W. Wachowicz, W. Wilczy\'nski, \emph{Multiplying balls in the space of continuous functions on $[0,1]$}. Studia Math. 170~(2) (2005), 203--209. 
\bibitem{Beh11a} E. Behrends, \emph{Walk the dog, or: products of open balls in the space of continuous functions}.  Func. Approx. Comment. Math.  44 (2011), 153--164.
\bibitem{Beh11} E. Behrends, \emph{Products of n open subsets in the space of continuous functions on $[0,1]$}.  Studia Math.  204~(1) (2011), 73--95.
\bibitem{Beh16} E. Behrends, \emph{Where is pointwise multiplication in real $CK$-spaces locally open?} Fund. Math., to appear.
\bibitem{Car05} N.L. Carothers, \emph{A Short Course on Banach Space Theory}. Cambridge University Press, Cambridge, 2005.
\bibitem{Eng89} R. Engelking, \emph{General Topology}. Heldermann Verlag, Berlin, 1989.
\bibitem{Kat51} M. Kat\v{e}tov, \emph{On real-valued functions in topological spaces}.  Fund.~Math.  38 (1951), 85--91. Correction: Fund. Math. 40 (1953), 203--205.
\bibitem{Kom06} A. Komisarski, \emph{A connection between multiplication in $C(X)$ and the dimension of $X$}. Fund.~Math.  189 (2006), no. 2, 149–-154. 
\bibitem{Kow10} S. Kowalczyk, \emph{Weak openness of multiplication in the space $C(0,1)$}. Real Analysis Exchange  35 (2009/2010), 235--242.
\bibitem{MS15} J.P. Moreno, R. Schneider, \emph{Some geometry of convex bodies in $C(K)$ spaces}. J.~Math. Pures Appl.  103~(9) (2015), 352--373.
\bibitem{Ton52} H. Tong, \emph{Some characterizations of normal and perfectly normal spaces}. Duke Math.~J. 19 (1952), 289--292.
\end{thebibliography}
\end{document}